\documentclass{amsart}

\RequirePackage{latexsym} \RequirePackage{amsthm}
\usepackage{enumerate}
\usepackage{dsfont}
\usepackage{mathrsfs}
\newtheorem{theorem}{Theorem}[section]

\newtheorem{lemma}[theorem]{Lemma}
\newtheorem{corollary}[theorem]{Corollary}

\theoremstyle{definition}
\newtheorem{definition}[theorem]{Definition}

\theoremstyle{remark}

\newtheorem{remark}[theorem]{Remark}

\numberwithin{equation}{section}

\begin{document}

%%
%% The title of the paper goes here.  Edit to your title.
%%

\title[Eigenvalue Inequalities via ABP Method]{Eigenvalue Inequalities for Fully Nonlinear Elliptic Equations via the Alexandroff-Bakelman-Pucci Method}

%%
%% Now edit the following to give your name and address:
%% 

\author{Dimitrios Gazoulis}
\address{Department of Mathematics, University of Athens (EKPA), Panepistemiopolis, 15784 Athens,
Greece}
\email{dgazoulis@math.uoa.gr}
%\urladdr{} % Delete if not wanted.

%%
%% If there is another author uncomment and edit the following.
%%

%\author{Second Author}
%\address{Department of Mathematics, University of South Carolina,
%Columbia, SC 29208}
%\email{second@math.sc.edu}
%\urladdr{www.math.sc.edu/$\sim$second}

%%
%% If there are three of more authors they are added in the obvious
%% way. 
%%

%%%
%%% The following is for the abstract.  The abstract is optional and
%%% if not used just delete, or comment out, the following.
%%%

%\date{}



\maketitle

%\author{Second author name \surname{Surname}}

%\address{Second author address \email{xxxx@xxxx.xxx.xx}}

\begin{abstract} In this work we establish eigenvalue inequalities for elliptic differential operators either for Dirichlet or for Robin eigenvalue problems, by using the technique introduced by Alexandroff, Bakelman and Pucci. These inequalities can be extended for fully nonlinear elliptic equations, such as for the Monge-Ampère equation and for Pucci's equations. As an application we establish a lower bound for the $ L^p -$norm of the Laplacian and this bound is sharp, in the sense that, when equality is achieved then a symmetry property is obtained. In addition, we obtain an $ L^{\infty} $ bound for the gradient of solutions to fully nonlinear elliptic equations and as a result, a $ C^3 $ estimate.

%These bounds can also be derived for vectorial elliptic equations, such as for Harmonic maps as we briefly present in the appendix.
\end{abstract}

%\keywords{}

%\classification{}

\section{Introduction}

In this paper, we present a bound for subsolutions of elliptic differential equations by using the technique introduced by Alexandroff, Bakelman and Pucci to establish their ABP estimate. This bound is in terms of the normal derivative in the boundary of a given domain. As a result we derive eigenvalue inequalities for elliptic differential operators either for Dirichlet or for Robin eigenvalue problems. These bounds can be extended to fully nonlinear equations, such as for the Monge-Ampère equation and for Pucci's equation.

The Alexandroff, Bakelman and Pucci estimate (or ABP estimate) is an $ L^{\infty} $ bound for subsolutions of the elliptic problem
\begin{equation}\label{ABPEllipticSubsolutionProblem}
\begin{gathered}
Lu \geq f \\
\textrm{where} \;\: Lu := \sum_{i,j}a_{ij}(x) \partial_{x_i x_j}u + \sum_i b_i(x) \partial_{x_i}u + c(x) u \;,
\end{gathered}
\end{equation} 
with $ a_{ij}, b_i $ bounded and measurable functions in a bounded domain $ \Omega \subset \mathbb{R}^n $ and $ c \leq 0 $. Then
\begin{equation}\label{ABPestimateStatement}
\sup_{\Omega} u \leq \sup_{\partial \Omega} u^+ + C \: diam(\Omega) \: || L u ||_{L^n(\Omega)} \;,
\end{equation}
for every function $ u \in C^2 ( \Omega) \cap C( \overline{\Omega}) $. The constant $ C $ depends only on the coefficients of $ L $ and $ diam(\Omega) $ denotes the diameter of $ \Omega $ (see section 2.5 in \cite{HL} for further details).

The ABP estimate has several extensions to more general problems, such as for viscosity solutions to fully nonlinear elliptic equations. It is also a basic tool in the regularity theory for fully nonlinear elliptic equations of the form
\begin{align*}
F( \nabla^2 u) =0 \;,
\end{align*}
see for instance \cite{CC}.

In this work, motivated by the classical techniques for proving the ABP estimate, together with the techniques in \cite{XCabre} and \cite{Cabre2},
we establish Theorem \ref{ThmNormalDerABPestimate}, that is an estimate of the form
\begin{equation}\label{NormalDerEstimate}
\inf_{\partial \Omega} \left| \frac{\partial u}{\partial \nu} \right| \leq C \: \left\Vert \frac{f^-}{D^*} \right\Vert_{L^n(\Omega)} \;,
\end{equation}
for every function $ u \in C^2 ( \Omega) \cap C^1( \overline{\Omega}) $ that satisfies \eqref{ABPEllipticSubsolutionProblem}. The constant $ C $ depends only on $ n, \; diam(\Omega) $ and the coefficients of $ L $ and $ D^* = (det(A))^{1/n} $.

Furthermore, X. Cabré in \cite{XCabre} and \cite{Cabre2} gave an elegant proof of the isoperimetric inequality using the ABP technique. The isoperimetric inequality says that among all domains of given volume the ball has the smallest perimeter. Utilizing similar arguments we obtain a lower bound for the $ L^n- $norm of the Laplacian in terms of the normal derivative on the boundary of a given domain, that is, Theorem \ref{ThmLowerBoundForSubharmonic}, which states that
\begin{equation}\label{LBLapl}
|| \Delta u ||_{L^n(\Omega)} \geq n \alpha  \: | B_1 |^{1/n} \;,
\end{equation}
whenever $ \dfrac{\partial u}{\partial \nu} \geq \alpha >0 $ on $ \partial \Omega $. In particular, if $ \dfrac{\partial u}{\partial \nu}= \alpha >0 $ on $ \partial \Omega $, we have
\begin{equation}\label{LBLapl_p}
|| \Delta u ||_{L^p(\Omega)} \geq \alpha \left( \frac{n^{n(p-1)} |B_1|^{p-1}}{|\partial \Omega|^{p-n}} \right)^{\frac{1}{p(n-1)}} \;\:,\;\: \textrm{for any} \;\: p \geq n,
\end{equation}
(see Theorem \ref{ThmLowerBoundForSubharmonic}) and by taking the limit as $ p \rightarrow + \infty $, this bound becomes
\begin{equation}\label{LBLapl_inf}
|| \Delta u ||_{L^{\infty}(\Omega)} \geq \alpha \left( \frac{n^n |B_1|}{|\partial \Omega|} \right)^{\frac{1}{n-1}}.
\end{equation}
A direct corollary of inequalities \eqref{LBLapl}, \eqref{LBLapl_p} and \eqref{LBLapl_inf} is a symmetry property when the equality in these estimates is achieved, in the spirit of the work of Serrin in \cite{Serrin}.

There are additional applications of the estimate in \eqref{LBLapl}, for instance, when it is applied to semilinear equations of the form
\begin{equation}\label{SemilinearEq}
\Delta u = f(u) \;,
\end{equation}
we obtain an $ L^{\infty} $ bound for the gradient of the solutions, see Theorem \ref{ThmLinftyBoundForGradient}.
We present these results in section 2, where we also obtain a $ C^3 $ estimate as application of this $ L^{\infty} $ bound. 
%So, if we consider in addition $ f(u) = \lambda u $
%%give examples and citations
%mention that radial symmetry is related to stable solutions

The estimate \eqref{NormalDerEstimate} can be established for the Monge-Ampère equation and we similarly derive eigenvalue inequalities. In particular, in section 3 we address the eigenvalue problem
\begin{equation}\label{MongeAmpEigenvIntro}
\begin{cases}
det ( \nabla^2 u) = \lambda | u |^n \;\;\;,\; \textrm{in}  \;\; \Omega \\ \;\;\;\;\;\;\; u = 0 \;\;\;\;\;\;\;\;\;\;\;\;\;\;,\;  \textrm{on} \;\; \partial \Omega
\end{cases}
\end{equation}
We note that also in \cite{Cabre1}, eigenvalue inequalities have been proved via the ABP method for the first eigenvalue of an elliptic operator.

Finally, in section 4, we extend some of these results to viscosity solutions of fully nonlinear elliptic equations. Particularly, Theorem \ref{ThmNormalDerABPfullyNonlinearOmega} is an estimate of the form \eqref{NormalDerEstimate}, from which we establish the analogous eigenvalue inequalities for both Dirichlet and Robin eigenvalue problems and in Theorem \ref{ThmL^inftyBoundonGradFullyNonlinear}, an $ L^{\infty} $ bound for the gradient of smooth solutions is derived. This $ L^{\infty} $ bound gives a $ C^3 $ estimate for the solutions as in the case of the Laplacian in section 2.
%and consider the eigenvalue problems of the form
%\begin{equation}\label{FullyNonlinearEqIntro}
%\begin{cases}
%F(x,u, \nabla^2 u) = \lambda | u |^p \;\;\;\;,\;\; \textrm{in}  \;\; \Omega \\ \;\;\;\;\;\;\;\;\;\; u = 0 \;\;\;\;\;\;\;\;\;\;\;\;\;\;\;\;,\;  \textrm{on} \;\; \partial \Omega
%\end{cases}
%\end{equation}

%Finally, in the appendix we derive similar bounds for systems of elliptic equations, such as for Harmonic maps.
$ \\ $

\section{A Bound for the normal derivative of subsolutions to Linear Elliptic equations}

Suppose $ \Omega $ is an open, bounded domain in $ \mathbb{R}^n $ and denote as $ \nu $ the unit normal pointing outwards. Throughout this section will denote as $ L $ a linear elliptic differential operator of the form
\begin{equation}\label{EllipticDifOperator}
Lu:= \sum_{i,j}a_{ij}(x) \partial_{x_i x_j}u + \sum_i b_i(x) \partial_{x_i}u + c(x) u \;,
\end{equation}
where $ a_{ij} \;,\: b_i $ and $ c $ are continuous in $ \Omega \;,\; c \leq 0 $ and $ L $ is uniformly elliptic in $ \Omega $ in the following sense:
\begin{equation}\label{EllipticityConditionforL}
\theta | \xi |^2 \leq \sum_{i,j} a_{ij}(x) \xi_i \xi_j \leq \Theta | \xi |^2 \;\;\;,\;\; \forall \;\: x \in \Omega \; \textrm{and any} \;\: \xi \in \mathbb{R}^n.
\end{equation}
for some positive constants $ \theta \leq \Theta $. Consider also $ D^* = (det(A))^{1/n} $ where $ A=(a_{ij}) $ is positive definite in $ \Omega $ and $ 0 < \theta \leq D^* \leq \Theta $.

We also introduce the concept of contact sets. For $ u \in C^1(\Omega) $ we define
\begin{equation}\label{LowerContactSetΓ_u}
\Gamma_u : = \lbrace x \in \Omega \: : \: u(y) \geq u(x) + \nabla u(x) \cdot (y-x) \;,\: \forall \; y \in \overline{\Omega} \rbrace \;.
\end{equation}

The set $ \Gamma_u $ is called the \textit{lower contact} set of $ u $. It is the set of points $ x $ such as the tangent hyperplane to the graph $ u $ at $ x $ lies below $ u $ in all of $ \overline{\Omega} . $

Similarly we define the \textit{upper contact} set of $ u $,
\begin{equation}\label{UpperContactSetΓ^u}
\Gamma^u := \lbrace x \in \Omega \: : \: u(y) \leq u(x) + \nabla u(x) \cdot (y-x) \;,\: \forall \; y \in \overline{\Omega} \rbrace \;,
\end{equation}
$ \Gamma^u $ is the set of points $ x $ such as the tangent hyperplane to the graph $ u $ at $ x $ lies above $ u $ in all of $ \overline{\Omega} . $ We also denote as $ B_r(x_0) $ (or $ B_r $), the ball of radius $ r$ centered at some point $ x_0$.

To begin with, we state the following property of $ C^1 $ functions. We rely on the arguments in the proof of the isoperimetric inequality in \cite{Cabre2}.

\begin{lemma}\label{Lemma1}
Let $ u : \Omega \rightarrow \mathbb{R} $ and $ u \in C^1( \overline{\Omega} ) $ such that 
\begin{equation}
m:= \inf_{\partial \Omega} \left| \dfrac{\partial u}{\partial \nu} \right| >0  \;.
\end{equation}
Then either
\begin{equation}\label{nablauGamma_u}
B_m(0) \subset \nabla u ( \Gamma_u) \;\;,\;\: \textrm{if} \;\: \frac{\partial u}{\partial \nu} \bigm|_{\partial \Omega} \: > 0 \;,
\end{equation}
or
\begin{equation}
B_m(0) \subset \nabla u ( \Gamma^u) \;\;,\;\: \textrm{if} \;\: \frac{\partial u}{\partial \nu} \bigm|_{\partial \Omega} \: < 0 \;.
\end{equation}
\end{lemma}
$ \\ $
\begin{proof}
Without loss of generality assume that $ \dfrac{\partial u}{\partial \nu} >0 $ on $ \partial \Omega $.

Let $ z \in B_m(0) $, where $ m= \inf_{\partial \Omega} \dfrac{\partial u}{\partial \nu} . \\ $
Consider the function
\begin{align*}
H(x) = u(x) - z \cdot x \;.
\end{align*}
So, there exists $ y \in \overline{\Omega} $ such that
\begin{equation}\label{ProofLemma1Eq1}
H (y) = \min_{x \in \overline{\Omega}} H(x) \;.
\end{equation}
This is, up to a sign, the Legendre transform of $ u $.

Now, we claim that $ y \in \Omega $. Indeed, if $ y \in \partial \Omega $, then 
\begin{equation}\label{ProofLemma1Eq2}
\frac{\partial H}{\partial \nu}(y) \leq 0 \;,
\end{equation}
that is,
\begin{equation}\label{ProofLemma1Eq3}
\frac{\partial u}{\partial \nu}(y) \leq z \cdot \nu \leq | z |< m \;,
\end{equation}
which contradicts the definition of $ m $.

Therefore $ y \in \Omega $, which means that $ H $ attains its minimum at an interior point of $ \Omega $.

Thus
\begin{equation}
\nabla u(y) = z \;\;\; \textrm{and} \;\: y \in \Gamma_u.
\end{equation}
\end{proof}
$ \\ $

\begin{lemma}\label{Lemma2} Suppose $ g \in L^1_{loc}(\mathbb{R}^n) $ is nonnegative. Then for any $ u \in C^2(\Omega) \cap C^1(\overline{\Omega}) $ such that $ m:= \inf_{\partial \Omega} \left| \dfrac{\partial u}{\partial \nu} \right| >0 $, there holds either
\begin{equation}\label{Lemma2Statement1}
\int_{B_m(0)} g \leq \int_{\Gamma_u} g(\nabla u) | det(\nabla^2 u) | \;\;,\;\: \textrm{if} \;\: \frac{\partial u}{\partial \nu} \bigm|_{\partial \Omega} \: > 0 \;,
\end{equation}
or
\begin{equation}\label{Lemma2Statement2}
\int_{B_m(0)} g \leq \int_{\Gamma^u} g(\nabla u) | det(\nabla^2 u) | \;\;,\;\: \textrm{if} \;\: \frac{\partial u}{\partial \nu} \bigm|_{\partial \Omega} \: < 0 \;.
\end{equation}
\end{lemma}
$ \\ $
\begin{proof}
The proof is similar to the proof of Lemma 2.24 in \cite{HL} by considering 
\begin{align*}
\tilde{M} = m = \inf_{\partial \Omega} \left| \dfrac{\partial u}{\partial \nu} \right|
\end{align*}
and utilizing Lemma \ref{Lemma1}.
\end{proof}
$ \\ $

\begin{theorem}\label{ThmNormalDerABPestimate}
Let $ u \in C^2(\Omega) \cap C^1(\overline{\Omega}) $ satisfies $ Lu \geq f $ in $ \Omega $ with the following conditions: $ \inf_{\partial \Omega} \left| \dfrac{\partial u}{\partial \nu} \right| >0 $ and $ \dfrac{|b|}{D^*} \:,\: \dfrac{f}{D^*} \in L^n(\Omega)$.

Then there holds that
\begin{equation}\label{NormalDerABPestimateStatement}
\inf_{\partial \Omega} \left| \dfrac{\partial u}{\partial \nu} \right| \leq C \left\Vert \frac{f^-}{D^*} \right\Vert_{L^n( \Omega)} \;,
\end{equation}
and $ C $ is a constant depending only on $ n \:,\; diam(\Omega) $ and $ \left\Vert \dfrac{b}{D^*} \right\Vert_{L^n(\Omega)} $.
\end{theorem}
$ \\ $
\begin{proof}
We sketch the proof in the case where $ f=c = 0 $, since it is similar to the proof of Theorem 2.21 in \cite{HL}. In our case, we set $ \tilde{M}= \inf_{\partial \Omega} \left| \dfrac{\partial u}{\partial \nu} \right| $.

The elliptic inequality $ Lu \geq f $ gives that 
\begin{align*}
 (- \sum_{i,j} a_{ij} \partial_{x_i x_j}u)^n \leq |b|^n | \nabla u |^n  \;.
\end{align*}
For any positive definite matrix $ A=(a_{ij}) $ it holds
\begin{align}\label{DetBoundByL}
det( \nabla^2 u) \leq \frac{1}{D} \left( \frac{- \sum_{i,j} a_{ij} \partial_{x_i x_j}u}{n} \right)^n \;\;\; \textrm{on} \;\: \Gamma^u \;\:,\; \textrm{where} \;\: D = det(A).
\end{align}
So, we consider $ g(z) = (| z |^n + \delta^n) ^{-n} \;,\; \delta>0 $ and then we apply Lemma \ref{Lemma2}. By letting $ \delta \rightarrow 0^+ $ we conclude.
\end{proof}

$ \\ $
\begin{remark}\label{RmkForThmABPnormalDer}
\textbf{(1)} There are various problems with Neumann boundary conditions of the form
\begin{align*}
\dfrac{\partial u}{\partial \nu} = g(x) \;\;,\; \textrm{on} \;\: \partial \Omega \;,
\end{align*}
with $ |g(x)| \geq c_0 >0 $ in which the assumption $ \inf_{\partial \Omega} \left| \dfrac{\partial u}{\partial \nu} \right| >0 $ is satisfied.
$ \\ $
%\textbf{(2)} The constant $ C $ can be written as
%$ C= exp \left[ \frac{2^{n-2}}{| B_1 | n^n} \left( || \dfrac{b}{D^*}||_{L^n(\Omega)} +1 \right) \right] -1 \; $ (see \cite{HL}). $ \\ $
\textbf{(2)} In particular, it holds that either
\begin{equation}\label{RmkForThmABPnormalDerStatement1}
\inf_{\partial \Omega} \left| \dfrac{\partial u}{\partial \nu} \right| \leq C \left\Vert \frac{f^-}{D^*} \right\Vert_{L^n( \Gamma^u)} \;\;,\;\: \textrm{if} \;\: \frac{\partial u}{\partial \nu} \bigm|_{\partial \Omega} \: < 0 \;,
\end{equation}
or
\begin{equation}\label{RmkForThmABPnormalDerStatement2}
\inf_{\partial \Omega} \left| \dfrac{\partial u}{\partial \nu} \right| \leq C \left\Vert \frac{f^-}{D^*} \right\Vert_{L^n( \Gamma_u)} \;\;,\;\: \textrm{if} \;\: \frac{\partial u}{\partial \nu} \bigm|_{\partial \Omega} \: > 0 \;.
\end{equation}

\end{remark}
$ \\ $

\subsection{A Lower bound for the norm of the Laplacian}

We will prove that the $ L^n -$norm of the Laplacian of a smooth function $ u $ is bounded from below by the normal derivative of $ u $ on $ \partial \Omega $. This estimate can be extended to an $ L^{p} $ and $ L^{\infty} $ lower bound of the norm when the normal derivative of the function is constant on the boundary of the domain. Again, the proof relies on the proof of Theorem 4.1 in \cite{Cabre2}.

\begin{theorem}\label{ThmLowerBoundForSubharmonic}
Let $ u \in C^2(\Omega) \cap C^1(\overline{\Omega}) $ be such that $ \dfrac{\partial u}{\partial \nu} \geq \alpha >0 $ on $ \partial \Omega $. Then
\begin{equation}\label{LowerBoundForSubharmonicEst1}
|| \Delta u ||_{L^n(\Omega)} \geq n \alpha  \: | B_1 |^{1/n} \;.
\end{equation}
In addition, if $ \dfrac{\partial u}{\partial \nu} = \alpha >0 $ on $ \partial \Omega $ and $ \Delta u \geq 0 $, we have
\begin{equation}\label{LowerBoundForSubharmonicEst2}
|| \Delta u||_{L^p(\Omega)} \geq \alpha \left( \frac{n^{n(p-1)} |B_1|^{p-1}}{|\partial \Omega|^{p-n}} \right)^{\frac{1}{p(n-1)}} \;,
\end{equation}
for any $ p \geq n $. In particular,
\begin{equation}\label{LowerBoundForL_infLapl}
|| \Delta u||_{L^{\infty}(\Omega)} \geq \alpha \left( \frac{n^n |B_1|}{|\partial \Omega|} \right)^{\frac{1}{n-1}}.
\end{equation}
\end{theorem}
\begin{proof}
By Lemma \ref{Lemma1}, we have that
\begin{equation}\label{ProofThmLowerBoundSubharmonicEq1}
B_{\alpha}(0) \subset \nabla u( \Gamma_u) \;.
\end{equation}
We now apply Lemma \ref{Lemma2} with $ g \equiv 1 $, so
\begin{equation}\label{ProofThmLowerBoundSubharmonicEq2}
| B_1 | \: \alpha^n \leq \int_{ \Gamma_u} det( \nabla^2 u) dx \;.
\end{equation}
Thus by the classical arithmetic geometric mean inequality applied to the non negative eigenvalues $ \lambda_1,..., \lambda_n $ of $ \nabla^2 u \;,\: x \in \Gamma_u $, we have
\begin{equation}\label{ProofThmLowerBoundSubharmonicEq3}
det( \nabla^2 u) \leq \left( \frac{tr( \nabla^2 u)}{n} \right)^n = \left( \frac{\Delta u}{n} \right)^n \;\;\;,\: x \in \Gamma_u \;.
\end{equation}
Therefore
\begin{equation}\label{ProofThmLowerBoundSubharmonicEq4}
| B_1 | \: \alpha^n \leq \frac{1}{n^n} \int_{ \Gamma_u} ( \Delta u)^n dx \leq \frac{1}{n^n} \int_\Omega | \Delta u |^n dx
\end{equation}
and we conclude
\begin{equation}\label{ProofThmLowerBoundSubharmonicEq5}
n \alpha \: | B_1 |^{1/n}  \leq \: || \Delta u||_{L^n( \Omega)} \;.
\end{equation}
$ \\ $

Next, for proving \eqref{LowerBoundForSubharmonicEst2}, assume that
\begin{equation}\label{ProofThmLowerBoundSubharmonicEq6}
\frac{\partial u}{\partial \nu} = \alpha >0 \;\;\; \textrm{on} \;\: \partial \Omega
\end{equation}
and by Hölder's inequality on $ || \Delta u||^n_{L^n( \Omega)} $,
\begin{equation}\label{ProofThmLowerBoundSubharmonicEq7}
\int_{\Omega} | \Delta u |^n dx = \int_{\Omega} | \Delta u |^s | \Delta u |^{n- s} dx \leq \left( \int_{\Omega} | \Delta u | \: dx \right)^s \left( \int_{\Omega} | \Delta u |^{\frac{n-s}{1-s}}dx \right)^{1-s} \;,
\end{equation}
for any $ s \in (0,1) $, which gives
\begin{equation}\label{ProofThmLowerBoundSubharmonicEq8}
n^n \alpha^n \: | B_1 | \leq \alpha^s | \partial \Omega |^s \left( \int_{\Omega} | \Delta u |^{\frac{n-s}{1-s}}dx \right)^{1-s} \;,
\end{equation}
by \eqref{LowerBoundForSubharmonicEst1}.

So, we set $ p := \dfrac{n-s}{1-s} $, $ p \in ( n , + \infty)\;,\: s $ is written as $ s= \dfrac{p-n}{p-1}$ and we get
\begin{equation}\label{ProofThmLowerBoundSubharmonicEq9}
|| \Delta u ||_{L^p(\Omega)}^{\frac{p(n-1)}{p-1}} \geq \frac{ n^n \alpha^{\frac{p(n-1)}{p-1}} | B_1| }{| \partial \Omega |^{\frac{p-n}{p-1}}} \;,
\end{equation}
which gives \eqref{LowerBoundForSubharmonicEst2}.

For \eqref{LowerBoundForL_infLapl}, we take the limit $ p \rightarrow + \infty $ in \eqref{LowerBoundForSubharmonicEst2} and we conclude.
\end{proof}

\begin{corollary}\label{CorLowerBoundForLaplacian}
Let $ u \in C^2(\Omega) \cap C^1(\overline{\Omega}) $ be such that $ \dfrac{\partial u}{\partial \nu} \bigm|_{\partial \Omega} \geq \alpha >0 $ and $ \Delta u \geq 0 $. Then
\begin{equation}\label{CorLowerBoundForLaplacianL_pEst}
|| \Delta u||_{L^p(\Omega)} \geq \left( \frac{(n \alpha)^{n(p-1)} |B_1|^{p-1}}{ (\int_{\partial \Omega} \frac{\partial u}{\partial \nu} dS)^{p-n}} \right)^{\frac{1}{p(n-1)}} \;,
\end{equation}
for any $ p \geq n $. In addition,
\begin{equation}\label{CorLowerBoundForLaplL_infLapl}
|| \Delta u||_{L^{\infty}(\Omega)} \geq \left( \frac{(n \alpha)^n | B_1|}{\int_{\partial \Omega} \frac{\partial u}{\partial \nu} dS} \right)^{\frac{1}{n-1}}.
\end{equation}
\end{corollary}

$ \\ $

A direct consequence of the Theorem \ref{ThmLowerBoundForSubharmonic} is a symmetry property for functions that satisfy the equality in \eqref{LowerBoundForSubharmonicEst1}. This result is in the spirit of Serrin's in \cite{Serrin}.

\begin{corollary}\label{CorollarySymmetryProperty}
Let $ u \in C^2(\Omega) \cap C^1(\overline{\Omega}) $ be such that $ \dfrac{\partial u}{\partial \nu} \geq \alpha >0 $ on $ \partial \Omega $ and assume that the equality holds in \eqref{LowerBoundForSubharmonicEst1}. 

Then
\begin{equation}\label{CorSymmetryPropEq}
u(x) = c_0 | x- x_0 |^2 \;\;\;,\;\: \textrm{and} \;\: \Omega \;\: \textrm{is a ball centered at} \;\: x_0.
\end{equation}
\end{corollary}
$ \\ $
\begin{proof}
When the equality in \eqref{LowerBoundForSubharmonicEst1} holds, we have equality in \eqref{ProofThmLowerBoundSubharmonicEq3}. This means that all the eigenvalues of $ \nabla^2 u $ are equal and thus $ \nabla^2 u = k I_n $, where $ k \in \mathbb{R} $ and $ I_n $ is the $ n \times n $ identity matrix. So we conclude that $ u(x) = c_0 | x- x_0 |^2 $ for some $ c_0 \in \mathbb{R} $ and some $ x_0 \in \mathbb{R}^n $.

Now we have that $ \Gamma_u = \Omega $ and $ B_{\alpha}(0) \subset \nabla u ( \Omega) = c_0 \Omega - c_0 x_0 . \\ $ 
Thus
\begin{equation}\label{proofCorSymmetryEq1}
B_{\frac{\alpha}{c_0}}(x_0) \subset \Omega \;\;\; \textrm{and} \;\; | B_{\frac{\alpha}{c_0}}(x_0) | = | \Omega |
\end{equation}
If there exists $ y_0 \in \Omega \setminus \overline{B}_{\frac{\alpha}{c_0}}(x_0) $, we would have $ B_\delta (y_0) \subset \Omega \setminus \overline{B}_{\frac{\alpha}{c_0}}(x_0) $ and contradicts the equality of the measures in \eqref{proofCorSymmetryEq1}.

So, $ \overline{\Omega} = \overline{B}_{\frac{\alpha}{c_0}}(x_0) $ and since $ \Omega $ is open, then $ \Omega $ is a ball centered at $ x_0. $
\end{proof}

$ \\ $

Note that the same symmetry property is obtained when the equality in \eqref{LowerBoundForSubharmonicEst2} or \eqref{LowerBoundForL_infLapl} holds. So, one question could be, if one of the inequalities \eqref{LowerBoundForSubharmonicEst1} and \eqref{LowerBoundForSubharmonicEst2} is ``close'' of being equality, whether $ u $ is ``close'' of being of the form \eqref{CorSymmetryPropEq}.

Moreover, another consequence of Theorem \ref{ThmLowerBoundForSubharmonic}, together with the Krylov-Safonov Harnack inequality (see \cite{KS1}, \cite{KS2}), is an $ L^{\infty} $ bound for the gradient of solutions to semilinear elliptic equations. $ \\ $

\begin{theorem}\label{ThmLinftyBoundForGradient}
Let $ u $ be a smooth solution of
\begin{equation}\label{ThmLinftyBoundForGradientSemilinearEq}
\Delta u = f(u) \;\;\;,\;\; x \in B_{R}
\end{equation}
such that $ u =0 $ on $ \partial B_R $, where $ f \in C^1( \mathbb{R}) $ and we extend $ u $ to be zero outside $ B_R $.
%and assume that either (i) $ | u | \leq M $ , or (ii) $ \sup_{t \in \mathbb{R}} | f'(t) | \leq \tilde{M} $ hold.

Then
\begin{equation}\label{ThmLinftyBoundForGradientStatementEq}
\sup_{B_R} | \nabla u |^2 \leq C \left( \frac{1}{n^2 | B_1 |^{2/n}} || f(u) ||^2_{L^n(B_R)} + R || \: | \nabla^2 u|^2 \: ||_{L^n(B_R)} \right)
\end{equation}
the constant $ C $ depends only on $ n $ and $ R^2 ||f'(u)||_{L^{\infty}(B_R)} $.
\end{theorem}
$ \\ $
\begin{proof}
Let $ P= | \nabla u |^2 $, we calculate
\begin{equation}\label{proofThmLinftyBoundForGradientEq1}
\Delta P = 2 | \nabla^2 u |^2 + 2 f'(u) | \nabla u |^2
\end{equation}
so $ P $ satisfies the elliptic equation $$ \frac{1}{2} \Delta P - f'(u) P = | \nabla^2 u |^2 , $$ where $ | \nabla^2 u |^2 = \sum_{1 \leq i,j \leq n} u_{x_i x_j}^2 $.

Consider $ \tilde{b} $ such that $ | f '(u)| \leq \tilde{b} $ for all $ x \in B_{2R} $ (for example, if $ f $ is globally Lipschitz with Lipschitz constant $ \tilde{M} $, we can take $ \tilde{b} = \tilde{M} $).

We then apply Krylov-Safonov Harnack inequality (see for example Theorem 4.1 in \cite{Cabre4}) and we get
\begin{equation}\label{proofThmLinftyBoundForGradientEq2}
\sup_{B_R} | \nabla u |^2 \leq C \left( \inf_{B_R} | \nabla u |^2 + R || \: | \nabla^2 u|^2 \: ||_{L^n(B_R)} \right) \;,
\end{equation}
where $ C $ depends only on $ n $ and $ R^2 ||f'(u)||_{L^{\infty}(B_R)} $.

Now w.l.o.g. assume that $ \inf_{B_R} | \nabla u |^2 >0 $, otherwise \eqref{ThmLinftyBoundForGradientStatementEq} holds. Since $ u = 0 $ on $ \partial B_R $, we have that $ \nabla u \perp \partial B_R $, which means that $ \nabla u $ and the unit normal of $ \partial B_R $ are parallel. Thus, $ \left| \dfrac{\partial u}{\partial \nu} \right| = | \nabla u | $ on $ \partial B_R $ and we get
\begin{equation}\label{proofThmLinftyBoundForGradientEq3}
\inf_{B_R} | \nabla u |^2 \leq \inf_{\partial B_R} | \nabla u |^2 = \left( \inf_{\partial B_R} \left| \dfrac{\partial u}{\partial \nu} \right| \right)^2 \leq \left( \frac{1}{n | B_1 |^{1/n}} || f(u) ||_{L^n(B_R)} \right)^2
\end{equation}
where for the last inequality we utilized Theorem \ref{ThmLowerBoundForSubharmonic} and \eqref{ThmLinftyBoundForGradientSemilinearEq}.
\end{proof}

$ \\ $

As an application of Theorem \eqref{ThmLinftyBoundForGradientStatementEq}, we have the following $ C^3 $ estimate

\begin{corollary}\label{CorollaryC3estimate}
Under the assumptions of Theorem \ref{ThmLinftyBoundForGradient}, it holds
\begin{equation}\label{CorollaryC3estimateEq}
\sup_{B_R} \sum_{i=1}^n | \Delta u_{x_i}|^2 \leq C \left( \frac{1}{n^2 | B_1 |^{2/n}} || f(u) ||^2_{L^n(B_R)} + R || \: | \nabla^2 u|^2 \: ||_{L^n(B_R)} \right)
\end{equation}
the constant $ C $ depends only on $ n \;,\: $ and $ R^2 ||f'(u)||_{L^{\infty}(B_R)} $.
\end{corollary}
\begin{proof}
By differentiating \eqref{ThmLinftyBoundForGradientSemilinearEq} we have
\begin{equation}\label{CorollaryC3estimatePfEq1}
| \Delta u_{x_i} | \leq || f'(u) ||_{L^{\infty}(B_R)} | u_{x_i}|
\end{equation}
and utilizing \eqref{ThmLinftyBoundForGradientStatementEq} we conclude.
\end{proof}
$ \\ $

At this point, we could ask whether the estimate in \eqref{ThmLinftyBoundForGradientStatementEq} holds without assuming that $ u $ vanishes on the boundary.

Another application is the following.

\begin{corollary}\label{CorollaryLowerBound}
Let $ u \in C^2(\Omega) \cap C^1(\overline{\Omega}) $ be a solution of
\begin{equation}\label{CorLowerBoundSemilinearEq}
\Delta u = f(u)
\end{equation}
such that $ f $ is locally Lipschitz and $ \inf_{\partial \Omega} \left| \dfrac{\partial u}{\partial \nu} \right| >0 $,
then
\begin{equation}\label{CorLowerBoundOnf}
|| f(u) ||_{L^n(\Omega)} \geq n \: | B_1 |^{1/n} \inf_{\partial \Omega} \left| \dfrac{\partial u}{\partial \nu} \right|
\end{equation}
If in addition $ \Omega = B_1 $ and $ u \bigm|_{\partial \Omega} =0 $, it holds
\begin{equation}\label{CorLowerBoundOnf2}
|| f(u) ||_{L^n( B_1)} \geq n \: | B_1 |^{1/n} | u'(1) |
\end{equation}
where $ u=u(r) $ is radially symmetric.
\end{corollary}
%$ \\ $
\begin{proof}
The inequality \eqref{CorLowerBoundOnf} is a direct consequence of Theorem \ref{ThmLowerBoundForSubharmonic}.
$ \\ $
Next, since $ \Omega $ is a unit ball we have that $ u $ is radially symmetric (see \cite{GNN} or \cite{HL}), that is, $ u = u(r) $ and $ \dfrac{\partial u}{\partial \nu} = \dfrac{\partial u}{\partial r} = u'(1) <0 $ on $ \partial \Omega $.
\end{proof}

$ \\ $

\begin{remark}\label{RmkCorollaryLowerBound}
\textbf{(1)} The symmetry property in Corollary \ref{CorollarySymmetryProperty} can be alternatively considered as uniqueness of the minimization problem
\begin{equation}\label{RmkCorSymMin}
\mathscr{J} (u) = \int_{B_1} | \Delta u |^n dx 
\end{equation}
subject to the constraint $ \dfrac{\partial u}{\partial \nu} \bigm|_{\partial B_1}  \geq \alpha >0 $, or even for the functionals $$ \mathscr{J}_p (u) = \int_{B_1} | \Delta u|^p dx \;\: \textrm{for}\;\: p>n \;,\; \mathscr{J}_{\infty} (u) = || \Delta u ||_{L^{\infty}(B_1)}, $$ when $ \dfrac{\partial u}{\partial \nu} \bigm|_{\partial B_1}  = \alpha >0 . \\ $
\textbf{(2)} In particular, if $ f(u) = c_0 | u |^{p/n} $ in Corollary \ref{CorollaryLowerBound} where $ p>0 $ and $ \dfrac{\partial u}{\partial \nu} \geq \alpha >0 $ on $ \partial \Omega $, then
\begin{equation}\label{RmkLowerBoundOnf2}
|| u ||_{L^p(\Omega)} \geq \left[ \frac{(n \alpha)^n \: | B_1 |}{c_0} \right]^{1/p} 
\end{equation}
Similarly, we can obtain other bounds by considering various examples.
%\begin{equation}\label{RmkCorLowerBoundEq}
%\Delta u = | \nabla u |^\beta \;\;\;,\;\: \beta \geq \frac{1}{n}
%\end{equation}
%then, \eqref{CorLowerBoundOnf} becomes
%\begin{equation}\label{RmkCorLowerBoundEstimate}
%|| \nabla u ||_{L^{\beta n}(\Omega)} \geq (n \alpha)^{1/ \beta} | B_1 |^{1/ n \beta}
%\end{equation}
\end{remark}

$ \\ $

\subsection{Applications: Lower bounds for eigenvalue problems}

\subsubsection{A Lower bound for eigenvalue problems of the Laplacian}

In this subsection we derive eigenvalue inequalities for the Laplacian for both the Dirichlet and the Robin eigenvalue problems. These inequalities are consequences of the Theorem \ref{ThmLowerBoundForSubharmonic}. We begin with the Dirichlet eigenvalue problem. $ \\ $

\begin{corollary}\label{CorEigenvalueIneqLaplacian}
Let $ \phi \in C^2( \Omega ) \cap C^1 (\overline{\Omega}) $ be the eigenfunction of the problem
\begin{equation}\label{CorEigenvIneqLaplEq}
\begin{cases} - \Delta \phi = \lambda \phi \;\;\;,\;\; \textrm{in} \;\: \Omega \\
\phi = 0 \;\;\;\;\;\;\;\;\;\;\;,\;\; \textrm{on} \;\: \partial \Omega
\end{cases}
\end{equation}
Then
\begin{equation}\label{CorEigenvIneqLaplStatement}
\lambda \geq n | B_1 |^{1/n} \frac{ \inf_{\partial \Omega} | \nabla \phi | }{|| \phi ||_{L^n(\Omega)} }
\end{equation}
\end{corollary}
\begin{proof}
By Theorem \ref{ThmLowerBoundForSubharmonic} we have
\begin{equation}\label{CorEigInLaplProofEq1}
\begin{gathered}
|| \Delta \phi ||_{L^n(\Omega)} \geq n | B_1 |^{1/n} \inf_{\partial \Omega} \left| \dfrac{\partial u}{\partial \nu} \right| \\
\Rightarrow | \lambda | \geq n | B_1 |^{1/n} \frac{\inf_{\partial \Omega} | \frac{\partial \phi}{\partial \nu} | }{|| \phi ||_{L^n(\Omega)}}
\end{gathered}
\end{equation}
Since $ u $ is zero on $ \partial \Omega $, it holds that $ \nabla u // \nu $, so
\begin{equation}\label{CorEigInLaplProofEq2}
\frac{\partial u}{\partial \nu} = | \nabla u \cdot \nu | = | \nabla u | \;\;\;,\; \textrm{on} \;\: \partial \Omega
\end{equation}
and we conclude.
\end{proof}

$ \\ $
\textbf{Note:} If $ \Omega =B_1 $, $ \phi $ is radially symmetric given explicitly by a Bessel function and $ \inf_{\partial \Omega} | \nabla \phi | = | \phi'(1) | $. If in addition $ n=3 $, we can write
\begin{align}\label{RmkBesselFunction}
\phi_1 (x) = \mu \sqrt{\frac{2}{\pi \lambda_1}} \frac{sin( \sqrt{\lambda_1} |x|)}{|x|} \;\;,\; \textrm{for some positive constant} \;\: \mu,
\end{align}
here we denote as $ \phi_1 $ the principal eigenfuction of the Laplacian and $ \lambda_1 $ the respective principal eigenvalue. Then we can calculate
\begin{align}\label{RmkEigenInLaplEq1}
|| \phi_1 ||_{L^3(\Omega)} \leq \mu \sqrt{\frac{2}{\pi}} | B_1 |^{1/3}
\end{align}
and inequality \eqref{CorEigenvIneqLaplStatement} can be written as
\begin{equation}\label{RmkEigenInLaplEq2}
\lambda_1^{3/2} \geq 3 | \sqrt{\lambda_1} cos(\sqrt{\lambda_1}) - sin(\sqrt{\lambda_1}) |
\end{equation}
By considering the function $ f(t) = \dfrac{sint}{t^3} - \dfrac{cos t}{t^2} \;,\; t >0 $ we can calculate the values $ t_1 < t_2 < t_3 $ where $ f(t) = \frac{1}{3} $ and observe that $ \lambda_1 \geq t_3^2 $.
% $ \simeq 9,1068 $. 
This bound, of course, is known since $ \lambda_1 $ is the square of the first root of the Bessel function in \eqref{RmkBesselFunction}. However, the lower bound in \eqref{CorEigenvIneqLaplStatement} can be generalized to fully nonlinear equations for both Dirichlet and Robin eigenvalue problems.
%%%references for all the above statements...

$ \\ $
\begin{remark}\label{RmkEigenInLapl}
%\textbf{(1)}
Let $ L_0 u = \sum_{i,j}a_{ij}(x) \partial_{x_i x_j}u + \sum_i b_i(x) \partial_{x_i}u $ be an elliptic operator and consider the eigenvalue problem
\begin{equation}\label{RmkEigenInBNVandCabre}
\begin{cases} - L_0 u_1 = \lambda_1 u_1 \;\:,\; \textrm{in} \;\: \Omega \\
\;\;\;\;\;\; u_1 = 0 \;\;\;\;\;\;\;,\;\; \textrm{on} \;\: \partial \Omega
\end{cases}
\end{equation}
In \cite{BNV}, Berestycki, Nirenberg and Varadhan proved the existence of a unique eigenvalue $ \lambda_1=\lambda_1(L_0, \Omega) >0 $ (i.e. the principal eigenvalue) having a smooth and positive eigenfunction $ u_1 $ that satisfies \eqref{RmkEigenInBNVandCabre}. Also, in Theorem 2.5 in \cite{BNV} they established the lower bound $$ \lambda_1 \geq C | \Omega |^{-2/n} $$ for some positive constant $ C $ depending only on $ n $, the ellipticity constants of $ L_0 $ and an upper bound on $ | \Omega |^{1/n} ||b||_{L^{\infty}(\Omega)} $. In \cite{Cabre1}, X. Cabré gave a simpler proof of this lower bound using the ABP method.
\end{remark}

$ \\ $

Next, we derive a similar bound for the Robin eigenvalue problem.

\begin{corollary}\label{CorEigenvalueIneqLaplRobin}
Let $ \phi \in C^2( \Omega ) \cap C^1 (\overline{\Omega}) $ be the eigenfunction of the problem
\begin{equation}\label{CorEigenvIneqLaplRobinEq}
\begin{cases} - \Delta \phi = \lambda \phi \;\;\;\;,\;\; \textrm{in} \;\: \Omega \\
\frac{\partial \phi}{\partial \nu} + \alpha \phi = 0 \;\;,\;\; \textrm{on} \;\: \partial \Omega
\end{cases}
\end{equation}
where $ \alpha \neq 0 $.

Then
\begin{equation}\label{CorEigenvIneqLaplRobinStatement}
| \lambda | \geq n | B_1 |^{1/n} | \alpha | \frac{ \inf_{\partial \Omega} | \phi | }{|| \phi ||_{L^n(\Omega)} }
\end{equation}
\end{corollary}
$ \\ $
\begin{proof}
We note that $ \inf_{\partial \Omega} | \frac{\partial \phi}{\partial \nu} | = | \alpha | \inf_{\partial \Omega} | \phi| $ and we conclude as in Corollary \ref{CorEigenvalueIneqLaplacian}.
\end{proof}

$ \\ $

\begin{remark}\label{RmkQuasiLinear}
By Theorem \ref{ThmLowerBoundForSubharmonic}, utilizing  \eqref{DetBoundByL}, the results in Corollaries \ref{CorEigenvalueIneqLaplacian} and \ref{CorEigenvalueIneqLaplRobin} can be extended for Quasi-linear operators of the form
\begin{equation}\label{Quasi-LinearEigenvalueEq}
\begin{cases} - \sum_{i,j} a_{ij}(x,u, \nabla u) u_{x_i x_j} = \lambda u \;\;\;,\; \textrm{in} \;\: \Omega \\
\;\;\;\;\;\;\;\;\;\;\;\;\;\; u =0 \;\;\;\;\;\;\;\;\;\;\;\;\;\;\;\;\;\;\;\;\;\;, \; \textrm{on} \;\: \partial \Omega
\end{cases}
\end{equation}
where $ a_{ij} $ satisfy the ellipticity condition
\begin{equation}\label{EllipticityConditionQuasiLinear}
\tilde{\theta} | \xi |^2 \leq \sum_{i,j} a_{ij}(x,u, \nabla u) \xi_i \xi_j \leq \tilde{\Theta} | \xi |^2
\end{equation}
for some constants $ 0< \tilde{\theta} \leq \tilde{\Theta} $.
\end{remark}

$ \\ $

\section{Eigenvalue inequalities for the Monge-Ampère equation}

We will now consider solutions of the Monge-Ampère equation
\begin{equation}\label{MongeAmpereEq}
det( \nabla^2 u) = f(x,u, \nabla u) \;\;\;,\;\; u : \Omega \subset \mathbb{R}^n \rightarrow \mathbb{R}
\end{equation}

The first estimate is similar to that of Theorem \ref{ThmNormalDerABPestimate}.

\begin{theorem}\label{ThmNormalDerABPMongeAmpere}
Let $ u \in C^2(\Omega) \cap C^1(\overline{ \Omega}) $ be a solution of \eqref{MongeAmpereEq} and assume that $ \inf_{\partial \Omega} \dfrac{\partial u}{\partial \nu} >0 $.

Then
\begin{equation}\label{NormalDerABPMongeAmpereStatement}
\inf_{\partial \Omega} \frac{\partial u}{\partial \nu}  \leq\left( \frac{|| f ||_{L^1(\Gamma_u)}}{|B_1|} \right)^{1/n}
\end{equation}
\end{theorem}
$ \\ $
\begin{proof}
With similar arguments as in the proof of Theorem \ref{ThmLowerBoundForSubharmonic} and applying Lemmas \ref{Lemma1}, \ref{Lemma2}, we obtain
\begin{equation}\label{ProofThmNormalDerABPMongeAmpereEq1}
\left( \inf_{\partial \Omega} \frac{\partial u}{\partial \nu} \right)^n | B_1 | \leq \int_{\Gamma_u} | det( \nabla^2 u) |dx
\end{equation}
and by \eqref{MongeAmpereEq} we conclude.
\end{proof}

$ \\ \\ $
\textbf{Note:} In the case where $ \inf_{\partial \Omega} \dfrac{\partial u}{\partial \nu} <0 $, it holds
\begin{equation}\label{NormalDerABPMongeAmpereStatement2note}
\sup_{\partial \Omega} \frac{\partial u}{\partial \nu} \geq - \left( \frac{|| f ||_{L^1(\Gamma^u)}}{|B_1|} \right)^{1/n}
\end{equation}
$ \\ $

We now introduce the eigenvalue problem for the Monge-Ampère operator $ det( \nabla^2 u) $ in an open, smooth, bounded and uniformly convex domain $ \Omega \subset \mathbb{R}^n $. This problem was first studied by Lions in \cite{Lions} where he showed that there exist a unique constant $ \lambda=\lambda( \Omega) >0 $ and a unique (up to positive multiplicative constant) nonzero convex function $ u \in C^{1,1} (\overline{\Omega}) \cap C^{\infty}(\Omega) $ solving the eigenvalue problem
\begin{equation}\label{EigenvalueProblemMongeAmpere}
\begin{cases}  det( \nabla^2 u)  = \lambda | u|^n \;\;\;\;\; \textrm{in} \;\: \Omega \\ u = 0 \;\;\;\;\;\;\;\;\;\;\;\;\;\;\;\;,\;\; \textrm{on} \;\: \partial \Omega
\end{cases}
\end{equation}
$ \\ $

\begin{corollary}\label{CorEigenvIneqMongeAmpere}
Let $ u \in C^2(\Omega) \cap C^1( \overline{\Omega}) $ be a solution to the eigenvalue problem \eqref{EigenvalueProblemMongeAmpere}.

Then
\begin{equation}\label{CorEigenvIneqMongeAmpereStatement}
\lambda \geq | B_1 | \left( \frac{\inf_{\partial \Omega} | \nabla u |}{||u||_{L^n(\Omega)}} \right)^n
\end{equation}
In addition, we have the following estimate
\begin{equation}\label{CorEigenvIneqMongeAmpereStatement2}
\inf_{\partial \Omega} | \nabla u | \leq  \frac{C(n)}{| \Omega |^{2/n}} ||u||_{L^n(\Omega)}
\end{equation}
for some constant $ C(n) $ depending only on the dimension $ n $.
\end{corollary}
$ \\ $
\begin{proof}
Without loss of generality assume that $ \inf_{\partial \Omega} \dfrac{\partial u}{\partial \nu} >0 $ (in the case where $ \inf_{\partial \Omega} \left| \dfrac{\partial u}{\partial \nu} \right| = 0 $, there is nothing to prove).

We apply Theorem \ref{ThmNormalDerABPMongeAmpere} and obtain
\begin{equation}\label{ProofCorEigenvIneqMongeAmpereEq1}
\lambda^{1/n} || u ||_{L^n(\Omega)} \geq | B_1 |^{1/n} \inf_{\partial \Omega} \dfrac{\partial u}{\partial \nu}
\end{equation}
Now since $ u=0 $ on $ \partial \Omega $, it holds that $ \nabla u \: // \: \nu $, so
\begin{equation}\label{ProofCorEigenvIneqMongeAmpereEq2}
\frac{\partial u}{\partial \nu}  = | \nabla u \cdot \nu | = | \nabla u |
\end{equation}
Thus, by \eqref{ProofCorEigenvIneqMongeAmpereEq1} and \eqref{ProofCorEigenvIneqMongeAmpereEq2} we get \eqref{CorEigenvIneqMongeAmpereStatement}. $ \\ $

Finally, for the estimate \eqref{CorEigenvIneqMongeAmpereStatement2}, we utilize the estimate for the eigenvalue problem in \cite{Le} (see Theorem 1.1), which states that
\begin{align*}
\lambda(\Omega) \leq \frac{C(n)}{| \Omega |^2}
\end{align*}
and we conclude.
\end{proof}
$ \\ $

\begin{remark}\label{RmkEigenvProblemMongeAmpere}
%\textbf{(1)} 
As in the previous section, we can consider the Robin eigenvalue problem
\begin{equation}\label{RmkEigenvProblemMongeAmpereRobin}
\begin{cases}  det( \nabla^2 u)  = \lambda | u|^n \;\;\;\;\; \textrm{in} \;\: \Omega \\ \frac{\partial u}{\partial \nu} + \alpha u = 0 \;\;\;\;\;\;\;\;,\;\; \textrm{on} \;\: \partial \Omega
\end{cases}
\end{equation}
and obtain
\begin{equation}\label{RmkEigenvProblemMongeAmpereRobinBound}
\lambda \geq | B_1 | \: | \alpha | \left( \frac{ \inf_{\partial \Omega} | u| }{|| u||_{L^n(\Omega)}} \right)^n
\end{equation}
%\textbf{(2)} The result of Corollary \ref{CorEigenvIneqMongeAmpere} can be extended to the problem
%\begin{equation}\label{Rmk1EigenvProblemMongeAmpereEq}
%\begin{cases}  det( \nabla^2 u)  = \lambda | u|^p \;\;\;\;\; \textrm{in} \;\: \Omega \\ u = 0 \;\;\;\;\;\;\;\;\;\;\;\;\;\;\;\;,\;\; \textrm{on} \;\: \partial \Omega
%\end{cases}
%\end{equation}
%\textbf{(2)} We could also consider the problem
%\begin{equation}\label{Rmk2EigenvProblemMongeAmpereEq}
%\begin{cases}  det( \nabla^2 u)  = \chi_A \;\;\;\;\; \textrm{in} \;\: \Omega \\ u = 0 \;\;\;\;\;\;\;\;\;\;\;\;\;\;\;\;,\;\; \textrm{on} \;\: \partial \Omega
%\end{cases}
%\end{equation}
%where $ \chi_A $ is the characteristic function of the set $ A $. In this case \eqref{NormalDerABPMongeAmpereStatement} can be written as
%\begin{equation}\label{Rmk2EigenvProblemMongeAmpereEq2}
%\inf_{\partial \Omega} | \nabla u | \leq \left( \frac{| \Omega \cap A |}{| B_1 |} \right)^{1/n}.
%\end{equation}
%Trudinger in \cite{Trudinger}, applied the ABP estimate to the problem \eqref{Rmk2EigenvProblemMongeAmpereEq} to establish a proof of the classical isoperimetric inequality.
\end{remark}

$ \\ $

\section{The case of Fully Nonlinear Elliptic Equations}

For the convenience of the reader we introduce Pucci's extremal operators and the respective spaces of viscosity subsolutions and supersolutions. We follow chapter 3 in \cite{CC}, where the ABP estimate is proved for viscosity solutions of fully nonlinear elliptic equations. These results are originated in \cite{C1}, \cite{C2}.

First we give the definition of viscosity solutions for fully nonlinear equations.

\begin{definition}\label{DefViscosity} Let $ \Omega \subset \mathbb{R}^n $. A continuous function $ u : \Omega \rightarrow \mathbb{R} $ is a viscosity subsolution (respectively supersolution) of the equation
\begin{equation}\label{DefFullyNonlinearEq}
F( \nabla^2 u(x),x) = f(x) \;\;\;,\; x \in \Omega
\end{equation}
when the following condition holds:

If $ x_0 \in \Omega \;,\: \phi \in C^2( \Omega) $ and $ u-\phi $ has a local maximum at $ x_0 $ then
\begin{equation}\label{DefViscSubsol}
F( \nabla^2 \phi (x_0),x_0) \geq f(x_0)
\end{equation}
(respectively if $ u- \phi $ has a local minimum at $ x_0 $ then $ F( \nabla^2 \phi (x_0),x_0) \leq f(x_0) $).

If $ u $ is both subsolution and supersolution we say that $ u $ is a viscosity solution of \eqref{DefFullyNonlinearEq}.
\end{definition}

We use the notation $ F( \nabla^2 u,x) = f(x) $ in the viscosity sense in $ \Omega $ whenever $ u $ is a viscosity solution (or $ \geq \:, \; \leq $ for subsolution, supersolution respectively).

Let now $ 0 < \vartheta \leq \varTheta $ and $ A \in \mathcal{S} $, where $ \mathcal{S} $ is the space of real $ n \times n $ symmetric matrices. We define
\begin{align}\label{PucciOperators}
\mathcal{M}^- (A, \vartheta , \varTheta) = \mathcal{M}^- (A) = \vartheta \sum_{e_i>0} e_i + \varTheta \sum_{e_i <0} e_i \\
\mathcal{M}^+ (A, \vartheta , \varTheta) = \mathcal{M}^+ (A) = \varTheta \sum_{e_i>0} e_i + \vartheta \sum_{e_i <0} e_i
\end{align}
where $ e_i = e_i(A) $ are the eigenvalues of $ A $ and $ \mathcal{M}^- \;,\: \mathcal{M}^+ $ are the Pucci's extremal operators.

It holds that $ \mathcal{M}^- $ and $ \mathcal{M}^+ $ are uniformly elliptic (see section 2 in \cite{CC}).

\begin{definition}\label{SpaceOfViscositySol}
Let $ f  $ be a continuous function in $ \Omega $ and $ 0< \vartheta \leq \varTheta . \\ $ (1) We denote as $ \mathscr{M}^+ ( \vartheta, \varTheta, f) $ the space of continuous functions $ u $ in $ \Omega $ such that $$ \mathcal{M}^+ ( \nabla^2 u, \vartheta, \varTheta) \geq f(x) $$ in the viscosity sense.
$ \\ $ (2) Similarly, we denote as $ \mathscr{M}^- ( \vartheta, \varTheta, f) $ the space of continuous functions $ u $ in $ \Omega $ such that $$ \mathcal{M}^- ( \nabla^2 u, \vartheta, \varTheta) \leq f(x) $$ in the viscosity sense. $ \\ $ (3) We set $ \mathscr{M} ( \vartheta, \varTheta, f) = \mathscr{M}^+ ( \vartheta, \varTheta, f) \cap \mathscr{M}^- ( \vartheta, \varTheta, f) $.
\end{definition}
$ \\ $

We will now prove the analog of Theorem \ref{ThmNormalDerABPestimate}.

\begin{theorem}\label{ThmNormalDerABPfullyNonlinear}
Let $ u \in \mathscr{M}^- ( \vartheta, \varTheta, f) \cap C^1( \overline{B}_r) $ in $ B_r $, where $ B_r \subset \mathbb{R}^n $ is an open ball of radius $ r $ and $ f $ is a continuous bounded function. Assume that $ \inf_{\partial \Omega} \left| \dfrac{\partial u}{\partial \nu} \right| >0 $.

Then
\begin{equation}\label{ThmNormalDerABPfullyNonlStatementEq}
\inf_{\partial \Omega} \left| \dfrac{\partial u}{\partial \nu} \right| \leq C \left( \int_{B_r \cap \Gamma_u} (f^+)^n \right)^{1/n}
\end{equation}
where $ C $ depends only on $ n \;,\: \vartheta \;,\: \varTheta $.
\end{theorem}
$ \\ $
\begin{proof}
The proof is similar to the proof of Theorem 3.2 in \cite{CC} by utilizing Lemma \ref{Lemma1}.

In this case, the radius of the ball $ B_r $ do not appear in the inequality \eqref{ThmNormalDerABPfullyNonlOmegaStatementEq}, since the term $ \dfrac{1}{r}\sup_{B_r} u^- $ is replaced by the term $ \inf_{\partial \Omega} \left| \dfrac{\partial u}{\partial \nu} \right| $. So the equation (3.7) in the proof of Lemma 3.4 in \cite{CC} becomes $$ C(n) \left( \inf_{\partial \Omega} \left| \dfrac{\partial u}{\partial \nu} \right| \right)^n \leq  | \nabla u  ( \Gamma_u) | = | \nabla \tilde{\Gamma} u ( \Gamma_u) | ,$$ where $ \tilde{\Gamma} u $ is the convex envelope of $ u $ (see Definition 3.1 in \cite{CC}). We also used the fact that since $ u \in C^1 $, the convex envelope of $ u $ equals $ u $ in the set $ \Gamma_u $ ($ \Gamma_u $ is defined in \eqref{LowerContactSetΓ_u}).
\end{proof}

$ \\ $

\begin{theorem}\label{ThmNormalDerABPfullyNonlinearOmega}
Let $ u \in \mathscr{M}^- ( \vartheta, \varTheta, f) \cap C^1( \overline{\Omega}) $ where $ \Omega $ is a bounded open domain of $ \mathbb{R}^n $ and $ f $ is a continuous bounded function. Assume that $ \inf_{\partial \Omega} \left| \dfrac{\partial u}{\partial \nu} \right| >0 $.

Then
\begin{equation}\label{ThmNormalDerABPfullyNonlOmegaStatementEq}
\inf_{\partial \Omega} \left| \dfrac{\partial u}{\partial \nu} \right| \leq C \left( \int_{\Omega \cap \Gamma_u} (f^+)^n \right)^{1/n}
\end{equation}
where $ C $ depends only on $ n \;,\: \vartheta \;,\: \varTheta $.
\end{theorem}
$ \\ $
\begin{proof}
The proof is similar to the proof of Theorem 3.6 in \cite{CC}.
\end{proof}

$ \\ $

Now, observe that Proposition 2.13 in \cite{CC} allow us to generalize the above estimate to any fully nonlinear elliptic equation.

\begin{corollary}\label{CorNormalDerABPfullyNonl}
Let $ u \in C^1( \overline{\Omega}) $ be a viscosity supersolution of
\begin{equation}\label{CorNormalDerABPfullyNonlEq}
F(\nabla^2 u) = f(x)
\end{equation}
where $ F $ is elliptic with constants $ \vartheta \:, \; \varTheta \;,\: F(0)=0 $ and $ f $ is a continuous bounded function. Assume that $ \inf_{\partial \Omega} \left| \dfrac{\partial u}{\partial \nu} \right| >0 $.

Then
\begin{equation}\label{CorNormalDerABPfullyNonlStatement}
\inf_{\partial \Omega} \left| \dfrac{\partial u}{\partial \nu} \right| \leq C \left( \int_{\Omega \cap \Gamma_u} (f^+)^n \right)^{1/n}
\end{equation}
where $ C $ depends only on $ n \;,\: \vartheta \;,\: \varTheta $.
\end{corollary}

\begin{proof}
The proof is a direct consequence of Theorem \ref{ThmNormalDerABPfullyNonlinearOmega} and Proposition 2.13 in \cite{CC}.
\end{proof}
$ \\ $

\subsection{An $ L^{\infty} $ bound for the gradient}

\begin{theorem}\label{ThmWeakHarnackIn}
Let $ u $ be a smooth solution of
\begin{equation}\label{ThmWeakHarnFulNonEq}
\begin{cases}
F( \nabla^2 u) =f(u) \;\: \textrm{in} \;\: B_R \\
\;\;\; u =0 \;,\: \textrm{on} \;\: \partial B_R
\end{cases}
\end{equation}
where $ F $ is elliptic with constants $ \vartheta \:, \; \varTheta \;,\: F(0)=0 \;\:,\;  f: \mathbb{R} \rightarrow \mathbb{R} $.

Then
\begin{equation}\label{ThmWeakHarnStatement}
\left( \frac{1}{| B_R |} \int_{B_R} | \nabla u |^{2 \varepsilon} dx \right)^{\varepsilon} \leq C \left[ \left( \int_{B_R \cap \Gamma_u} (f^+(u))^n dx \right)^{1/n} +2 \varTheta R || \: | \nabla^2 u| \: ||^2_{L^n(B_{2R})} \right] 
\end{equation}
where $ C=C(n, \vartheta, \varTheta, \tilde{b} R^2) \;,\: \tilde{b} $ is such that $ | f'(u) | \leq \frac{\tilde{b}}{2} $ and $ \varepsilon>0 $ small.
\end{theorem}
$ \\ $
\begin{proof}
Arguing as in Lemma 2.4 in \cite{DG}, we have that
\begin{equation}\label{PfWeakHarnackEq1}
\sum_{i,j} d_{ij} P_{x_i x_j} -2f'(u) P \leq 2 \varTheta | \nabla^2 u |^2
\end{equation}
where $ P = | \nabla u |^2 $ and $ d_{ij} = \dfrac{\partial F}{\partial a_{ij}} $.

Then by the weak Harnack inequality (see Theorem 4.3 in \cite{Cabre4})
\begin{equation}\label{PfWeakHarnackEq2}
\left( \frac{1}{| B_R |} \int_{B_R} | \nabla u |^{2 \varepsilon} dx \right)^{\varepsilon} \leq C \left[ \inf_{B_R}| \nabla u |^2 +2 \varTheta R || \: | \nabla^2 u| \: ||^2_{L^n(B_{2R})} \right] 
\end{equation}

Now, by Corollary \ref{CorNormalDerABPfullyNonl} it holds,
\begin{equation}\label{PfWeakHarnackEq3}
\inf_{\partial \Omega} \left| \dfrac{\partial u}{\partial \nu} \right| \leq C_1 \left( \int_{B_R \cap \Gamma_u} (f^+)^n dx \right)^{1/n}
\end{equation}
where $ C_1=C_1(n,\vartheta, \varTheta) $.

Finally, since, $ u $ vanish on $ \partial B_R $, $ \left| \dfrac{\partial u}{\partial \nu} \right| = | \nabla u | $ on $ \partial B_R $ and we conclude \eqref{ThmWeakHarnStatement}.
\end{proof}

$ \\ $

\begin{theorem}\label{ThmL^inftyBoundonGradFullyNonlinear}
Let $ u $ be a smooth function that satisfies \eqref{ThmWeakHarnFulNonEq} and $ F \;,\: f $ as in Theorem \ref{ThmWeakHarnackIn} and assume that $ f'(u) \geq 0 $.

Then
\begin{equation}\label{ThmL^inftyBoundonGradStatement}
\sup_{B_R} | \nabla u |^2 \leq C \left\lbrace \left( \int_{B_R \cap \Gamma_u} (f^+)^n dx \right)^{1/n} +2 \varTheta R || \: | \nabla^2 u| \: ||^2_{L^n(B_{2R})} \right\rbrace
\end{equation}
where $ C=C(n, \vartheta, \varTheta, \tilde{b} R^2) $ and $ \tilde{b} $ is such that $ | f'(u)| \leq \frac{\tilde{b}}{2} $.
\end{theorem}
$ \\ $
\begin{proof}
By Theorem 5.1 in \cite{DG} we have
\begin{equation}\label{PfThmL^inftyGradFulNonlEq1}
\sup_{B_R} | \nabla u |^2 \leq \frac{C}{| B_{2R} |^{1/ p}} || \nabla u ||^2_{L^{2p}(B_R)}
\end{equation}
for $ p>0 $ small.

Therefore we take $ p=\varepsilon $ in Theorem \ref{ThmWeakHarnackIn} and we conclude.
\end{proof}

So, as in section 2, we have the following $ C^3 $ estimate as an application of Theorem \ref{ThmL^inftyBoundonGradFullyNonlinear}.

\begin{corollary}\label{CorC3EstimateFullyNonlinear}
Under the assumptions of Theorem \ref{ThmL^inftyBoundonGradFullyNonlinear}, if in addition $ f'(u) \;,\: | \nabla u| $ do not vanish in $ \overline{B}_R $ and  $ | \Delta u_{x_k} | \leq C_0 \;,\: k=1,..,n $, then
\begin{equation}\label{CorC3EstimateFullyNonlinearEquation}
\sup_{B_R} \sum_k | \Delta u_{x_k} |^2 \leq C \left\lbrace \left( \int_{B_R \cap \Gamma_u} (f^+)^n dx \right)^{1/n} +2 \varTheta R || \: | \nabla^2 u| \: ||^2_{L^n(B_{2R})} \right\rbrace
\end{equation}
where $ C=C(n, \vartheta, \varTheta, \tilde{b} R^2) $ and $ \tilde{b} $ is such that $ | f'(u) | \leq \frac{\tilde{b}}{2} $.
\end{corollary}
\begin{proof}
Differentiating \eqref{ThmWeakHarnFulNonEq}, we have
\begin{equation}\label{CorC3EstProofEq1}
\sum_{i,j} F_{a_{ij}}( \nabla^2 u) u_{x_i x_j x_k} = f'(u) u_{x_k} \;\:,\; k=1,...,n.
\end{equation}
Suppose for contradiction that for any number $ 0< \delta < \frac{1}{2} $, there exists $ x_0 \in B_R $ such that
\begin{equation}\label{CorC3EstProofEq2}
| \sum_{i,j} F_{a_{ij}}( \nabla^2 u(x_0)) u_{x_i x_j x_k}(x_0) | < \delta | \Delta u_{x_k}(x_0) |
\end{equation}
then we can deduce a sequence of numbers $ \delta_m \rightarrow 0 $ and a sequence of $ x_m \in B_R $ such that 
\begin{equation}\label{CorC3EstProofEq3}
| \sum_{i,j} F_{a_{ij}}( \nabla^2 u(x_m)) u_{x_i x_j x_k}(x_m) | < \delta_m | \Delta u_{x_k} (x_m)| \leq \delta_m C_0 \rightarrow 0
\end{equation}
and by \eqref{CorC3EstProofEq1},
\begin{equation}\label{CorC3EstProofEq4}
|f'(u(x_m)) u_{x_k}(x_m)| \rightarrow 0
\end{equation}
but $ x_m \rightarrow x_0 \in \overline{B}_R $ up to subsequence, thus
\begin{equation}\label{CorC3EstProofEq5}
\begin{gathered}
f'(u(x_m)) \rightarrow f'(u(x_0)) \;\; \textrm{and} \; u_{x_k}(x_m) \rightarrow u_{x_k}(x_0) \\
\textrm{up to subsequence}.
\end{gathered}
\end{equation}
so, $ f'(u(x_0)) =0 $ or $ u_{x_k}(x_0) =0 $ which contradicts the assumption on $ f' \;,\: | \nabla u| $.

Therefore there exist $ \delta >0 $ such that
\begin{equation}\label{CorC3EstProofEq6}
| \sum_{i,j} F_{a_{ij}}( \nabla^2 u) u_{x_i x_j x_k}| \geq \delta | \Delta u_{x_k}| \;\;,\; \textrm{uniformly in} \;\: B_R
\end{equation}
and utilizing \eqref{CorC3EstProofEq1} we get
\begin{equation}\label{CorC3EstProofEq6}
\delta^2 \sup_{B_R} \sum_k | \Delta u_{x_k}|^2 \leq \frac{\tilde{b}}{2} \sup_{B_R} | \nabla u |^2
\end{equation}
and we conclude by Theorem \ref{ThmL^inftyBoundonGradFullyNonlinear}.
\end{proof}

Note that in Corollary \ref{CorC3EstimateFullyNonlinear}, the assumption that $ f'(u) \;,\: | \nabla u| $ do not vanish in $ \overline{B}_R $, it suffices to be satisfied in the closure of the set of points where \eqref{CorC3EstProofEq2} is satisfied. So in the case where $ F( \nabla^2 u) = \Delta u $, this set is empty, and also $ \Delta u_{x_k} $ is bounded. In this view, Corollary \ref{CorC3EstimateFullyNonlinear} is the generalization of Corollary \ref{CorollaryC3estimate} for fully nonlinear elliptic equations.
$ \\ $

As in the case of Theorem \ref{ThmLinftyBoundForGradient} we could ask whether the estimates in Theorems \ref{ThmWeakHarnackIn} and \ref{ThmL^inftyBoundonGradFullyNonlinear} hold without a priori assume that $ u $ vanish on $ \partial B_R $.

$ \\ $

\subsection{Applications to Eigenvalue problems}

In this last subsection we consider the Dirichlet and the Robin eigenvalue problems for fully nonlinear equations and apply Theorem \ref{ThmNormalDerABPfullyNonlinear} to establish eigenvalue inequalities similar to the ones in the previous sections.

Let $ \Omega $ be a bounded open domain of $ \mathbb{R}^n $, $ p \geq 1 $ and consider the Dirichlet eigenvalue problem
\begin{equation}\label{DirichletEigenvalueFullyNonlinear}
\begin{cases}
\mathcal{M}^- ( \nabla^2 u, \vartheta, \varTheta) = \lambda | u |^p \;\;,\; \textrm{in} \;\: \Omega \\ \;\;\;\;\;\;\;\;\;\;\;\;\; u = 0 \;\;\;\;\;\;\;\;\;\;\:\;\;\;\;\;,\; \textrm{on} \;\: \partial \Omega
\end{cases}
\end{equation}
and the Robin eigenvalue problem
\begin{equation}\label{RobinEigenvalueFullyNonlinear}
\begin{cases}
\mathcal{M}^- ( \nabla^2 u, \vartheta, \varTheta) = \lambda | u |^p \;\;,\; \textrm{in} \;\: \Omega \\ \;\;\;\;\;\;\;\;\; \frac{\partial u}{\partial \nu} + \alpha u = 0 \;\;\;\;\:\;\;\;\;\;,\; \textrm{on} \;\: \partial \Omega
\end{cases}
\end{equation}
respectively.

We could also replace $ \lambda | u |^p $ by $ \lambda u $, so that in the case where $ \vartheta = \varTheta =1 $ the operator $ \mathcal{M}^- $ becomes the Laplace operator and we have the eigenvalue problem of the Laplace operator as a special case.

The Corollaries below give some bounds for the problems \eqref{DirichletEigenvalueFullyNonlinear} and \eqref{DirichletEigenvalueFullyNonlinear} respectively.
$ \\ $

\begin{corollary}\label{CorDirichletEigenFullyNonlinear}
Let $ u \in \mathscr{M}^- ( \vartheta, \varTheta, \lambda | u |^p ) \cap C^1( \overline{\Omega}) $ be a viscosity supersolution of \eqref{DirichletEigenvalueFullyNonlinear} and assume that $ \inf_{\partial \Omega} \left| \dfrac{\partial u}{\partial \nu} \right| >0 $.

Then
\begin{equation}\label{CorDirichletEigenFullyNonlStatementEq}
\lambda \geq \frac{ \inf_{\partial \Omega} \left| \dfrac{\partial u}{\partial \nu} \right| }{C || u^p ||_{L^n(\Omega \cap \Gamma_u)}}
\end{equation}
\end{corollary}

$ \\ $

\begin{corollary}\label{CorRobinEigenFullyNonlinear}
Let $ u \in \mathscr{M}^- ( \vartheta, \varTheta, \lambda | u |^p ) \cap C^1( \overline{\Omega}) $ be a viscosity supersolution of \eqref{RobinEigenvalueFullyNonlinear} and assume that $ \inf_{\partial \Omega} \left| u \right| >0 $.

Then
\begin{equation}\label{CorRobinEigenFullyNonlStatementEq}
\lambda \geq \frac{| \alpha| \inf_{\partial \Omega} | u | }{ C || u^p ||_{L^n(\Omega \cap \Gamma_u)}}
\end{equation}
\end{corollary}
$ \\ $

The proof of the above Corollaries are direct consequences of Theorem \ref{ThmNormalDerABPfullyNonlinearOmega}.

$ \\ $

\begin{remark}\label{RmkEigenvIneqGeneralFullyNonlinear} Note that Corollaries \ref{CorDirichletEigenFullyNonlinear} and \ref{CorRobinEigenFullyNonlinear} can be generalized to the eigenvalue problems
\begin{equation}\label{RmkDirichletEigenvalueGenFullyNonlinear}
\begin{cases}
F ( \nabla^2 u) = \lambda | u |^p \;\;,\; \textrm{in} \;\: \Omega \\ \;\;\;\;\;\;\;\;\;\;\;\;\; u = 0 \;\;\;\;\;,\; \textrm{on} \;\: \partial \Omega
\end{cases}
\end{equation}
and the Robin eigenvalue problem
\begin{equation}\label{RmkRobinEigenvalueGenFullyNonlinear}
\begin{cases}
F ( \nabla^2 u ) = \lambda | u |^p \;\;,\; \textrm{in} \;\: \Omega \\ \;\;\;\;\; \frac{\partial u}{\partial \nu} + \alpha u = 0 \;\;\;,\; \textrm{on} \;\: \partial \Omega
\end{cases}
\end{equation}
respectively, via Corollary \ref{CorNormalDerABPfullyNonl}. Here $ F $ is elliptic with ellipticity constants $ \vartheta \;,\: \varTheta $ and $ F(0)=0 $. Additionally, as noted previously, the right hand side for the eigenvalue problems \eqref{RmkDirichletEigenvalueGenFullyNonlinear} and \eqref{RmkRobinEigenvalueGenFullyNonlinear} could be slightly modified and obtain similar bounds.
\end{remark}

$ \\ \\ $

\textbf{Acknowledgments:} I wish to thank professor N. Alikakos for suggesting me to read the elegant proof of isoperimetric inequality proved by X. Cabré, which motivated a part of this work. Also, I would like to thank professor X. Cabré and professor N. Le for their comments in a previous version of this work that led to various improvements. Finally, the author acknowledges the ``Basic research Financing'' under the National Recovery and Resilience Plan ``Greece 2.0'' funded by the European Union-NextGeneration EU (H.F.R.I. Project Number: 016097).

%$ \\ $
%\begin{appendix}
%\section{Appendix: Extensions to elliptic systems of equations}
%
%\end{appendix}
%%APPENDIX: INCLUDE LEMMAS FOR THE PROOF OF THEOREM 4.3 in [CC]
%$ \\ \\ $
%\textbf{Ackowledgements:} I would like to thank professor N. Alikakos for suggesting me to read the elegant proof of isoperimetric inequality proved by X. Cabre, which motivated a part of this work.
%%++ acknowledge research program

$ \\ $

\end{document}